\documentclass{birkjour}
 \newtheorem{thm}{Theorem}[section]
 
 \newtheorem{lem}[thm]{Lemma}
 
 \theoremstyle{definition}
 \newtheorem{defn}[thm]{Definition}
 \theoremstyle{remark}
 \newtheorem{rem}[thm]{Remark}
 \newtheorem*{ex}{Example}
 \numberwithin{equation}{section}

\allowdisplaybreaks
\usepackage{dsfont}
\usepackage{overpic}

\usepackage{hyperref}

\newcommand{\clifford}{\mathcal{C}\ell}

\begin{document}

%
%
%
%
%
%
%
%
%

\title[A Clifford algebraic Approach to Line Geometry]{A Clifford algebraic Approach to\\Line Geometry}

\author[Daniel Klawitter]{Daniel Klawitter}

\address{%
Dresden University of Technology\\
Department of Mathematics\\
Institute of Geometry\\
Zellescher Weg 12 - 14\\
01062 Dresden\\
Germany}

\email{daniel.klawitter@tu-dresden.de}



\date{October 23, 2013}

\begin{abstract}
In this paper we combine methods from projective geometry, Klein's model, and Clifford algebra. We develop a Clifford algebra whose Pin group is a double cover of the group of regular projective transformations. The Clifford algebra we use is constructed as homogeneous model for the five-dimensional real projective space $\mathds{P}^5(\mathds{R})$ where Klein's quadric $M_2^4$ defines the quadratic form. We discuss all entities that can be represented naturally in this homogeneous Clifford algebra model.
Projective automorphisms of Klein's quadric induce projective transformations of $\mathds{P}^3(\mathds{R})$ and vice versa. Cayley-Klein geometries can be represented by Clifford algebras, where the group of Cayley-Klein isometries is given by the Pin group of the corresponding Clifford algebra. Therefore, we examine the versor group and study the correspondence between versors and regular projective transformations represented as $4\times 4$ matrices. Furthermore, we give methods to compute a versor corresponding to a given projective transformation.
\end{abstract}

\maketitle

\section{Introduction}
Homogeneous Coordinates for lines of three-dimensional projective space $\mathds{P}^3(\mathds{R})$ were introduced by Pl\"ucker, see \cite{blaschke:kinematikundquaternionen}. The Pl\"ucker coordinates of all lines form a quadric in five-dimensional projective space, the so-called Klein quadric denoted by $M_2^4\subset\mathds{P}^5(\mathds{R})$, see \cite{weiss}.
The group of regular projective transformations of $\mathds{P}^3(\mathds{R})$ is isomorphic to the group of projective automorphisms of Klein's quadric $M_2^4$, see \cite{potwal}. Moreover, the group of automorphic collineations of Klein's quadric is the isometry group of the Cayley-Klein space given by $\mathds{P}^5(\mathds{R})$ together with absolute figure $\mathrm{M}_2^4$. This isometry group corresponds to the Pin group of a special homogeneous Clifford algebra model, see \cite{gunn:kinematics}. Therefore, we introduce the Clifford algebra $\clifford_{(3,3)}$ constructed over the quadratic space $\mathds{R}^{(3,3)}$ and describe how points on Klein's quadric are embedded as null vectors. We use the approach carried out in \cite{hongbo}, where this construction was described for the first time. Furthermore, we discuss how geometric entities that are known from Klein's model can be transferred to this homogeneous Clifford algebra model.
The action of grade-$1$ elements corresponds to the action of null polarities on $\mathds{P}^3(\mathds{R})$. The main focus of this contribution is the representation of the occurring transformations as homogeneous matrices that act on $\mathds{P}^3(\mathds{R})$. Moreover, we prove that every regular projective transformation of $\mathds{P}^3(\mathds{R})$ can be expressed as the product of at most six null polarities, {\it i.e.}, skew-symmetric $4\times 4$ matrices.\\
\section{Klein's Model and its Clifford Algebra Representation}
We recall Klein's model of line space and give a homogeneous Clifford algebra model corresponding to Klein's model, cf. \cite{hongbo}. Moreover, we introduce the representation of various geometric entities of Klein's model in the corresponding Clifford algebra setting.
\subsection{Klein's Model}
Lines in three-dimensional space form a four-dimensional manifold called Klein's quadric, or Pl\"ucker's quadric. It is a special Grassmann variety. In fact, every straight line in $\mathds{P}^3(\mathds{R})$ is mapped to a point in five-dimensional projective space $\mathds{P}^5(\mathds{R})$, see \cite{blaschke:kinematikundquaternionen}. To make this explicit, we think of $\mathds{R}^4$ with its standard basis as model for $\mathds{P}^3(\mathds{R})$. A line is spanned by two different points\linebreak $X=x\mathds{R}$ and $Y=y\mathds{R}$ with homogeneous linearly independent coordinate vectors $x=(x_0,x_1,x_2,x_3)^\mathrm{T}$ and $y=(y_0,y_1,y_2,y_3)^\mathrm{T}$.
\begin{defn}\label{def1}\index{Pl\"ucker coordinates}
For two distinct points $X=x\mathds{R}=(x_0,x_1,x_2,x_3)^\mathrm{T}\mathds{R}$ and $Y=y\mathds{R}=(y_0,y_1,y_2,y_3)^\mathrm{T}\mathds{R}\in\mathds{P}^3(\mathds{R})$ we define the \emph{Pl\"ucker coordinates} of the line spanned by $X$ and $Y$ as:
\[p=(p_{01}\!:p_{02}\!:p_{03}\!:p_{23}\!:p_{31}\!:p_{12}), \mbox{ with }p_{ij}=\left|\begin{array}{cc} \!\!x_i \!\!&\!\! x_j\!\!\\ \!y_i \!\!&\!\! y_j\!\! \end{array}\right|.\]
The mapping $\mu:\mathcal{L}^3\mapsto\mathds{P}^5(\mathds{R})$, where $\mathcal{L}^3$ is the set of lines of $\mathds{P}^3(\mathds{R})$, that maps each line $L\in\mathcal{L}^3$ to a point $P=(p_{01}\!:\ldots:p_{12})$ of $\mathds{P}^5(\mathds{R})$ is called the \emph{Klein mapping}.
\end{defn}
\noindent Only those points of $\mathds{P}^5(\mathds{R})$ correspond to a line in $\mathds{P}^3(\mathds{R})$ that are contained in Klein's quadric $M_2^4$. Its equation is derived by expanding the determinant $\det(x,y,x,y)$ by complementary $2\times 2$ minors:
\[\left|\begin{array}{cccc}\!\! x_0 \!\!&\!\! x_1 \!\!&\!\! x_2 \!\!&\!\! x_3\!\!\\\!\!y_0 \!\!&\!\! y_1 \!\!&\!\! y_2 \!\!&\!\! y_3\!\!\\\!\!x_0 \!\!&\!\! x_1 \!\!&\!\! x_2 \!\!&\!\! x_3\!\!\\\!\!y_0 \!\!&\!\! y_1 \!\!&\!\! y_2 \!\!&\!\! y_3\!\! \end{array}\right|=\left|\begin{array}{cc}\!\! x_0 \!\!&\!\! x_1\!\!\\\!\! y_0 \!\!&\!\! y_1\!\!\end{array}\right|\left|\begin{array}{cc}\!\! x_2 \!\!&\!\! x_3\!\!\\\!\! y_2 \!\!&\!\! y_3\!\!\end{array}\right|+\left|\begin{array}{cc}\!\! x_0 \!\!&\!\! x_2\!\!\\\!\! y_0 \!\!&\!\! y_2\!\!\end{array}\right|\left|\begin{array}{cc}\!\! x_3 \!\!&\!\! x_1\!\!\\\!\! y_3 \!\!&\!\! y_1\!\!\end{array}\right|+\left|\begin{array}{cc}\!\! x_0 \!\!&\!\! x_3\!\!\\\!\! y_0 \!\!&\!\! y_3\!\!\end{array}\right|\left|\begin{array}{cc}\!\! x_1 \!\!&\!\! x_2\!\!\\\!\! y_1 \!\!&\!\! y_2\!\!\end{array}\right|=0,\]
which yields
\begin{equation}\label{pluecker}
M_2^4:p_{01}p_{23}+p_{02}p_{31}+p_{03}p_{12}=0.
\end{equation}
\begin{rem}
An algebraic variety with dimension $d$ and degree $k$ is denoted by a capital letter with subscript $k$ and superscript $d$, say $V_k^d$. Thus, Klein's quadric $M_2^4$ is an algebraic variety of degree two and dimension four. 
\end{rem}
\noindent Eq. \eqref{pluecker} can be reformulated as 
\begin{equation}\label{quadricmatrix}
x^\mathrm{T}\mathrm{Q}x=0, \mbox{ with }\mathrm{Q}=\begin{pmatrix}
 \mathrm{O} & \frac{1}{2}\mathrm{I}\\\frac{1}{2}\mathrm{I}&\mathrm{O}
\end{pmatrix},\, x=(x_0,x_1,x_2,x_3,x_4,x_5)^\mathrm{T},
\end{equation}
where $\mathrm{O}$ is the $3\times 3$ zero matrix and $\mathrm{I}$ is $3\times 3$ identity matrix. 
We have a one-to-one correspondence between lines in three-dimensional projective space and points on this quadric. The polarity of Klein's quadric $\nu$ can be expressed in matrix form by $\mathrm{Q}$. When working in five-dimensional projective space we use $\nu$ and postfix notation. Moreover, the polarity of Klein's quadric $\nu$ is regular, since
\[\det\mathrm{Q}=-\frac{1}{64}\neq 0.\]
The bilinear form induced by this polarity is denoted by $\Omega$ and defined by \[\Omega(x,x):=x^\mathrm{T}\mathrm{Q}x.\]
Two lines $L_1=l_1\mathds{R}$ and $L_2=l_2\mathds{R}$ given in Pl\"ucker coordinates intersect if, and only if, their images under the Klein map satisfy $\Omega(L_1,L_2):=l_1^\mathrm{T}\mathrm{Q}l_2=0$. This means, they have to be conjugate with respect to Klein's quadric. Now we ask for projective automorphisms of Klein's quadric induced by collineations or correlations in $\mathds{P}^3(\mathds{R})$. First, we transfer projective transformations acting on $\mathds{P}^3(\mathds{R})$ to automorphic collineations of $M_2^4$. Let
$\mathrm{C}= (c_{kl}), k,l=0,\dots,3$ be the matrix representation of a collineation. We apply this collineation to the points $X=x\mathds{R},\,Y=y\mathds{R}\in\mathds{P}^3(\mathds{R})$ with $x=(x_0,x_1,x_2,x_3)^\mathrm{T}$, $y=(y_0,y_1,y_2,y_3)^\mathrm{T}$ and compute the Pl\"ucker coordinates of the line joining $x'=Cx$ and $y'=Cy$.
The Pl\"ucker coordinates of the image line under this collineation are given by:
\begin{align*}
p_{ij}'&=\left|\begin{array}{cc}\!\! x'_i \!\!&\!\! x'_j\!\!\\\!\! y'_i \!\!&\!\! y'_j\!\!\end{array}\right|=x_i'y_j'-x_j'y_i'\\
&=\left( \sum\nolimits_{k}{c_{ik}x_k}\right) \left( \sum\nolimits_{l}{c_{jl}y_l}\right) -
  \left( \sum\nolimits_{l}{c_{jl}x_l}\right) \left( \sum\nolimits_{k}{c_{ik}y_k}\right)\\
  &=  \sum\nolimits_{k,l}{c_{ik}c_{jl}(x_ky_l-x_ly_k)},
\end{align*}
where $(i,j)$ is one of $(0,1),\,(0,2),\,(0,3),\,(2,3),\,(3,1)$ or $(1,2)$, see \cite[p. 139]{potwal}.
If we write the action of this transformation on the space of lines as matrix vector product we get a $6\times 6$ matrix $\mathrm{L}$ containing the coefficients from the equations above
\begin{equation*}
(p_{01}',\,p_{02}',\,p_{03}',\,p_{23}',\,p_{31}',\,p_{12}')^\mathrm{T}=
\mathrm{L}\cdot(p_{01},\,p_{02},\,p_{03},\,p_{23},\,p_{31},\,p_{12})^\mathrm{T}.
\end{equation*}
When we repeat this procedure for a correlation the columns of the matrix correspond to plane coordinates. Hence, we can compute the collineation in the image space in the same way, but in this case we have to calculate the Pl\"ucker coordinates of the image lines by the intersection of two planes instead of the connection of two points. We get
\begin{equation}\label{collineationlinespace}
(p_{01}',\,p_{02}',\,p_{03}',\,p_{23}',\,p_{31}',\,p_{12}')^\mathrm{T}=
\bar{\mathrm{L}}\cdot(p_{01},\,p_{02},\,p_{03},\,p_{23},\,p_{31},\,p_{12})^\mathrm{T},
\end{equation}
where $\bar{\mathrm{L}}$ defines an automorphic collineation of $M_2^4$ corresponding to a correlation in $\mathds{P}^3(\mathds{R})$.
\subsection{The homogeneous Clifford Algebra Model corresponding to\\Line Geometry}
General introductions to Clifford algebras can be found for example in \cite{gallier:cliffordalgebrascliffordgroups,garling:cliffordalgebras,perwass} and \cite{porteous:cliffordalgebrasandtheclassicalgroups}. Its connection to Cayley-Klein spaces is developed in \cite{gunn:kinematics}. In this article we will not recall Cayley-Klein geometries. An exhaustive treatise of this topic can be found in \cite{giering:vorlesungenueberhoeheregeometrie} or \cite{onishchik}. To build up the homogeneous model we use the quadratic form of Klein's quadric $M_2^4$, that is given by
\begin{equation*}
\mathrm{Q}=\begin{pmatrix} \mathrm{O} & \mathrm{I} \\ \mathrm{I} & \mathrm{O} \end{pmatrix},
\end{equation*}
where $\mathrm{O}$ is the $3\times 3$ zero matrix and $\mathrm{I}$ the $3 \times 3$ identity matrix. The matrix $\mathrm{Q}$ that is used here corresponds to the polarity of Klein's quadric, since multiplication with real scalars has no effect, see Eq.\ \eqref{quadricmatrix}. As underlying vector space for the Clifford algebra we take $\mathds{R}^6$ as vector space model for $\mathds{P}^5(\mathds{R})$. The corresponding Clifford algebra has signature $(p,q,r)=(3,3,0)$ (cf. \cite{hongbo}) and is of dimension $2^6=64$. Lines of $\mathds{P}^3(\mathds{R})$ represented by Pl\"ucker coordinates, see Definition \ref{def1}, correspond to null vectors in this algebra, {\it i.e.}, vectors that square to zero. A vector is given by
\begin{equation*}
\mathfrak{v}=x_1 e_1 + x_2 e_2 + x_3 e_3 + x_4 e_4 +x_5 e_5 +x_6 e_6
\end{equation*}
and its square is computed by
\begin{equation}\label{square}
\mathfrak{vv} = 2(x_1x_4+x_2x_5+x_3x_6).
\end{equation}
Eq. \eqref{square} evaluates to zero if, and only if, the Pl\"ucker condition \eqref{pluecker} is fulfilled, {\it i.e.}, if the point $X=(x_1,\dots,x_6)^\mathrm{T}\mathds{R}\in\mathds{P}^5(\mathds{R})$ is contained by $M_2^4$, and therefore, describes a line in $\mathds{P}^3(\mathds{R})$. The norm of a vector equals
\begin{equation*}
\mathfrak{vv}^\ast = -2(x_1x_4+x_2x_5+x_3x_6),
\end{equation*}
where $\mathfrak{v}^\ast$ denotes the conjugate element of $\mathfrak{v}$. Conjugation is defined by its action on generators:
\[(e_{i_1\dots i_k})^\ast=(-1)^ke_{i_k\dots i_1},\]
with $0<i_1<\dots<i_k\leq n=\dim V$.
Now we examine geometric entities that are described in this geometric algebra by inner product and outer product null spaces, see \cite{perwass}. Therefore, we define:
\begin{defn}
\begin{enumerate}
\item[(i)]A $k$-blade $\mathfrak{A}\in{\bigwedge}^k V$ is the $k$-fold exterior product of vectors $\mathfrak{v}_i,\,i=1,\dots,k$
\[\mathfrak{A}=\mathfrak{v}_1\wedge\dots\wedge\mathfrak{v}_k.\]
A $k$-blade that squares to zero is called a \emph{null $k$-blade}.
\item[(ii)] The \emph{inner product} and the \emph{outer product null space} of a $k$-blade\linebreak $\mathfrak{A}\in{\bigwedge}^k V$ are defined by
\[
\mathds{NI}(\mathfrak{A})\!=\!\left\lbrace\! \mathfrak{v} \!\in\!{\bigwedge}^1 V\!:\mathfrak{v}\cdot \mathfrak{A}\!=\!0\right\rbrace\!,\quad
\mathds{NO}(\mathfrak{A})\!=\!\left\lbrace\! \mathfrak{v} \!\in\!{\bigwedge}^1 V\!:\mathfrak{v}\wedge \mathfrak{A}\!=\!0\right\rbrace\!.
\]
\end{enumerate}
\end{defn}
\noindent  Furthermore, we have the property of the outer product and inner product null space:
\begin{equation}\label{nicap}
\mathds{NI}(\mathfrak{a}\wedge\mathfrak{b})=\mathds{NI}(\mathfrak{a})\cap\mathds{NI}(\mathfrak{b}),\quad
\mathds{NO}(\mathfrak{a}\wedge\mathfrak{b})=\mathds{NO}(\mathfrak{a})\oplus\mathds{NO}(\mathfrak{b}).
\end{equation}
The polarity of the metric quadric is given by multiplication with the pseudo-scalar $\mathfrak{J}:=e_{123456}$. The duality between subspaces of $\mathds{P}^5(\mathds{R})$ induced by the polarity is expressed by multiplication with the pseudo-scalar. Outer product null spaces can be used to describe the point set corresponding to an algebra element. Moreover, the dual geometric entity with respect to the polarity $\mathrm{Q}$ is obtained with the inner product null space
\[\mathds{NI}(\mathfrak{A})=\mathds{NO}(\mathfrak{AJ}).\]
\subsection{Subspaces contained in the Quadric and the corresponding\\Clifford Algebra Representation}\label{subspaces}
We are interested in the structure of Klein's quadric. The quadric is of hyperbolic type with two-dimensional generator spaces. Therefore, we take a closer look to one- and two-dimensional subspaces entirely contained in the quadric. In this paragraph we follow \cite{potwal} for each subspace, and immediately afterwards give the Clifford algebra description.
\vspace{5pt}
\paragraph{One-dimensional subspaces of $M_2^4$} To classify one-dimensional subspaces we give a lemma from \cite[Section 2.1.3]{potwal} without proof. Nevertheless, it is easy to verify that:
\begin{lem}
The Klein mapping takes a pencil of lines to a straight line contained in $M_2^4$. Vice versa, two points $X=x\mathds{R}$ and $Y=y\mathds{R}$ of $M_2^4$ correspond to intersecting lines in $\mathds{P}^3(\mathds{R})$ if, and only if, their span is contained in $M_2^4$.
\end{lem}
\paragraph{Lines contained in $M_2^4$ as null two-blades}
Null blades can be used to describe subspaces contained entirely in $M_2^4$. Subspaces that are contained in $M_2^4$ are either lines or two-spaces in $\mathds{P}^5(\mathds{R})$. A null two-blade generated by the exterior product of two null vectors corresponding to conjugate points on $M_2^4$ defines a line in $M_2^4\subset\mathds{P}^5(\mathds{R})$. Its outer product null space is the set of all null vectors corresponding to a pencil of lines in $\mathds{P}^3(\mathds{R})$.
\vspace{5pt}
\paragraph{Two-dimensional subspaces of $M_2^4$}
There are two types of two-spaces that are completely contained by $M_2^4$. The first one is given by the image of a  field of lines, {\it i.e.}, all lines contained by the same plane in $\mathds{P}^3(\mathds{R})$. The second type is given by all lines concurrent to the same point, {\it i.e.}, a bundle of lines. Now we ask for the intersection of two-spaces contained in Klein's quadric. Two-spaces of the same type always intersect in one point. This means two different fields of lines or two different bundles of lines contain one common line. Two two-spaces of different type may have either empty intersection or they intersect in a line. Thus, a field of lines and a bundle of lines may have empty intersection or they intersect in a common pencil of lines. The type of the two-space, {\it i.e.}, if it corresponds to a bundle of lines or a field of lines can be determined by its intersection with the Klein image of the field of ideal lines denoted by $P_\omega$. Note that $P_\omega$ corresponds to a field of lines, and therefore, to a two-space entirely contained in Klein's quadric. If a two-space corresponds to a bundle of lines it has either no point or a whole line in common with $P_\omega$. A two-space corresponding to a field of lines has one point in common with $P_\omega$. Figure \ref{fig:kleinquadric} shows the correspondence between lines and two-spaces on Klein's quadric and their pre-images under the Klein mapping. This figure is inspired by a figure from \cite[p. 142]{potwal}.
\begin{figure}%
\begin{center}
\begin{overpic}[width=0.7\columnwidth]{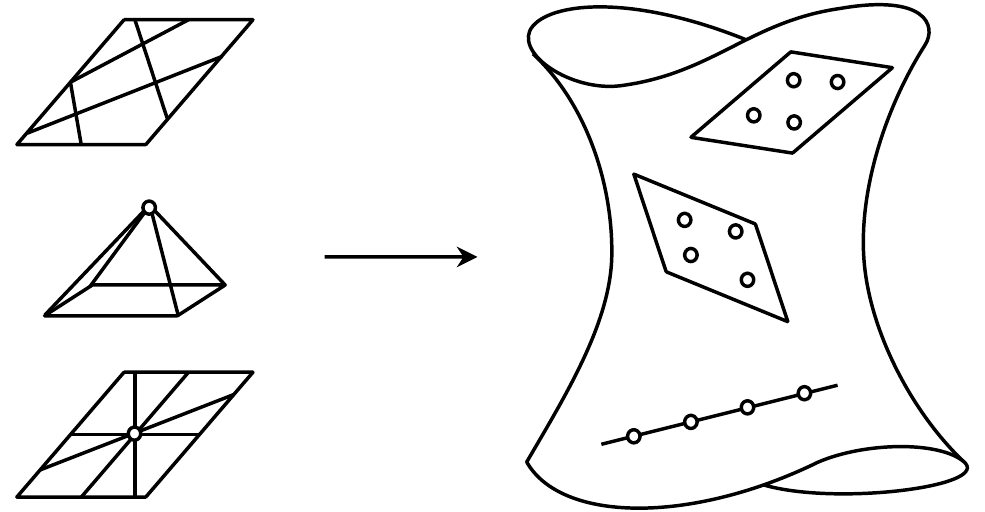}
	\put(39,27){$\mu$}
	\put(5,45){$F$}
	\put(5,25){$B$}
	\put(5,9){$P$}
	\put(85,47){$F\mu$}
	\put(79,25){$B\mu$}
	\put(64,10){$P\mu$}
	\put(60,46){$M_2^4$}
	\put(94,51){$\mathds{P}^5(\mathds{R})$}
	\put(0,51){$\mathds{P}^3(\mathds{R})$}
\end{overpic}
\caption{Subspaces on Klein's quadric and their geometric interpretation in $\mathds{P}^3(\mathds{R})$. A pencil of lines $P$ is mapped to the line $P\mu$ on $M_2^4$. A bundle of lines $B$ and a field of lines $F$ are mapped to a two-space $B\mu$ and $F\mu$ that are contained entirely in $M_2^4$.}%
\label{fig:kleinquadric}%
\end{center}
\end{figure}
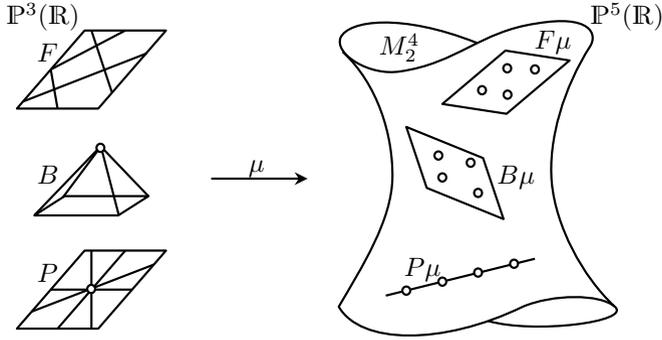
\paragraph{Two-spaces contained in $M_2^4$ as null three-blades}
A two-space $P_1^2$ in $\mathds{P}^5(\mathds{R})$ that is contained entirely in Klein's quadric can be expressed as the exterior product of three null vectors corresponding to points contained in the two-space. This results in a null three-blade. Its outer product null space consists of all null vectors that correspond to points contained in the two-space \linebreak $P_1^2\subset M_2^4$, {\it i.e.}, a bundle of lines or a field of lines.
\subsection{Subspace Intersections of $M_2^4$}
With Klein's mapping sets of lines can be studied as sets of points in five-dimensional projective space. Clearly, all points on the quadric are self-conjugate with respect to the polarity $\nu$. It is natural to ask for the intersections of Klein's quadric with subspaces. In Klein's model $k$-space intersections, $k\leq 4$ with $M_2^4\subset\mathds{P}^5(\mathds{R})$ define special sets of lines in $\mathds{P}^3(\mathds{R})$. These sets of lines correspond to so-called linear line manifolds, cf. \cite{potwal}. In this section we recall the occurring linear line manifolds and introduce their Clifford algebra representation as outer product null space of $(k+1)$-blades. Using duality $k$-space intersections can also be described as inner product null space of $(n-(k+1))$-blades where $n=6$ is the dimension of the vector space model of $\mathds{P}^5(\mathds{R})$.
\vspace{5pt}
\paragraph{Conics on Klein's quadric}
Two-spaces which are neither tangent to nor belong to Klein's quadric intersect the quadric in a conic section.
Tangent two-space intersections with $M_2^4$ result in degenerated conics, {\it i.e.}, two intersecting lines on Klein's quadric that correspond to two pencils of lines with one line in common in three-dimensional projective space. In this paragraph we discus the non-degenerated case. Therefore, we choose three skew lines with corresponding Pl\"ucker coordinates $L_1,L_2$ and $L_3$ with $\Omega(L_i,L_j)\neq 0,\, i,j=1,2,3,\, i\neq j$. These three lines possess image points spanning a two-space 
\[P_1^2(\alpha:\beta:\gamma)=\alpha L_1+\beta L_2+\gamma L_3,\,\alpha,\beta,\gamma\in\mathds{R}.\]
The set of all points that are conjugate to this two-space is spanned by\linebreak $\bar{P}_1^2=L_1\nu\cap L_2\nu \cap L_3\nu$. This defines the polar two-space of $P_1^2$. Since this polar two-space is the intersection of the three tangent hyperplanes $L_i\nu,\,i=1,\dots,3$, we conclude that the intersection of $\bar{P}_1^2\cap M_2^4$ corresponds to the set of all lines in $\mathcal{L}^3$ intersecting $L_1,L_2$, and $L_2$. We can do this construction for three points of the resulting conic again and get the same statement for the original points $L_1,L_2$, and $L_3$. Furthermore, we derive that these two sets of lines are different sets of generators of the same ruled surface. We define:
\begin{defn}\label{regulus}
The set of all lines intersecting three given mutually skew lines $L_1,L_2,L_3\in\mathcal{L}^3$ is called a \emph{regulus}. A regulus is part of a ruled quadric. The image of a regulus under the polarity $\nu$ that defines the second family of generators of a ruled quadric is called the \emph{opposite regulus}.
\end{defn}
\noindent Therefore, every regular conic on Klein's quadric corresponds to a regulus. Reguli can be distinguished in the Klein model by their intersection with the two-space $P_\omega$ corresponding to the field of ideal lines. In affine space a regulus that carries an ideal line is a hyperbolic paraboloid. Otherwise if there is no ideal line the conic corresponds to a hyperboloid of one sheet.
\vspace{5pt}
\paragraph{Conics on Klein's quadric as non-null three-blades}
Three-blades corresponding to two-spaces in $\mathds{P}^5(\mathds{R})$ can be defined as exterior product of three null vectors $\mathfrak{v}_1,\,\mathfrak{v}_2,\,\mathfrak{v}_3\in\bigwedge^1 V$ corresponding to points on $M_2^4$. If the three-blade $\mathfrak{B}\in\bigwedge^3 V$ squares to zero it corresponds to a two-space that is entirely contained in $M_2^4$, else it corresponds to a two-space that intersects in a conic on $M_2^4\subset\mathds{P}^5(\mathds{R})$. All points contained by this two-space can be computed with the help of the outer product null space
\[\mathds{NO}(\mathfrak{v}_1\wedge \mathfrak{v}_2\wedge\mathfrak{v}_3)=\mathds{NO}(\mathfrak{v}_1)\oplus\mathds{NO}(\mathfrak{v}_2)\oplus\mathds{NO}(\mathfrak{v}_3).\]
To get the null vectors located in the two-space $\mathfrak{p}=\alpha \mathfrak{v}_1+\beta \mathfrak{v}_2 +\gamma \mathfrak{v}_3$ we determine the zero divisors by $\mathfrak{pp}=0$.
This results in a quadratic equation involving the coefficients $\alpha,\,\beta$, and $\gamma$. The solution is given by the intersection of Klein's quadric with the two-space. In $\mathds{P}^5(\mathds{R})$ the dual of a two-space is a two-space and the points contained by the dual of a two-space can be calculated by the inner product null space of a three-blade corresponding to the two-space.
\vspace{5pt}
\paragraph{Linear Line Congruences}
The two-parametric set of lines corresponding to a three-space intersection of Klein's quadric is called a linear line congruence. In line geometry we distinguish between hyperbolic, parabolic, and elliptic linear line congruences, see \cite{potwal}.
\vspace{5pt}
\paragraph{Linear line congruences as four-blades}
Three-spaces are polar to lines. Thus, linear line congruences can be described by inner product null spaces of two-blades that correspond to lines in $\mathds{P}^5(\mathds{R})$ or outer product null spaces of four-blades that correspond to three-spaces in $\mathds{P}^5(\mathds{R})$. Lines in $\mathds{P}^5(\mathds{R})$ are represented by the exterior product of two vectors $\mathfrak{v}_1,\mathfrak{v}_2\in\bigwedge^1 V$ corresponding to points in $\mathds{P}^5(\mathds{R})$. A general line in $\mathds{P}^5(\mathds{R})$ written as outer product of two arbitrary vectors 
\begin{align*}
\mathfrak{v}_1&=x_1 e_1 + x_2 e_2 + x_3 e_3 +x_4 e_4+x_5 e_5 +x_6 e_6,\\
\mathfrak{v}_2&=y_1 e_1 + y_2 e_2 + y_3 e_3 +y_4 e_4+y_5 e_5 +y_6 e_6
\end{align*}
has the form
\begin{equation*}
\mathfrak{L}=\mathfrak{v}_1\wedge\mathfrak{v}_2=\sum\limits_{\substack{i,j=1\\i<j}}^6{\left|\begin{array}{cc}\!\! x_i \!\!&\!\! x_j\!\!\\\!\! y_i \!\!&\!\! y_j\!\!\end{array}\right|e_{ij}}.
\end{equation*}
With Eq. \eqref{nicap} we know that the two conditions 
\begin{align*}
&a_1x_4+a_2x_5+a_3x_6+a_4x_1+a_5x_2+a_6x_3=0,\\
&a_1y_4+a_2y_5+a_3y_6+a_4y_1+a_5y_2+a_6y_3=0
\end{align*}
are sufficient to describe the inner product null space of $\mathfrak{L}$.
All null vectors\linebreak $\mathfrak{a}=a_1e_1+a_2e_2+a_3e_3+a_4e_4+a_5e_5+a_6e_6$ satisfying these conditions correspond to the set of lines contained in a linear line congruence. In a dual way, a linear line congruence can also be represented as the outer product null space of a four-blade that can be constructed as exterior product of four linearly independent vectors.
\vspace{5pt}
\paragraph{Linear Line Complexes}
Linear line complexes correspond to hyperplanar intersections of $M_2^4$. Therefore, they define three-dimensional line manifolds.
\begin{defn}\label{linecomplex}
A \emph{linear line complex} is defined by all lines corresponding to the inner product null space of a vector $\mathfrak{c}\in\bigwedge^1 V$:
\[\mathds{NI}(\mathfrak{c})=\lbrace \mathfrak{l}\in{\bigwedge}^1V \mid\mathfrak{c}\cdot\mathfrak{l}=0 \rbrace.\]
In five-dimensional projective space we use the coordinates $L=(l_1,\dots,l_6)^\mathrm{T}\mathds{R}$ corresponding to $\mathfrak{l}$ and $C=(c_1,\dots,c_6)^\mathrm{T}\mathds{R}$ corresponding to $\mathfrak{c}$.
\end{defn}
\noindent A linear line complex is called \emph{singular} if the hyperplane is tangent to $M_2^4$ else \emph{regular}.
Note that the hyperplane defined by $P_1^4=\mathds{R}(\bar{c},c)^\mathrm{T}$ is polar to the point $C=(c,\bar{c})^\mathrm{T}\mathds{R}$. This point is contained in $M_2^4$ in the singular case.
\begin{rem}
We can derive any regular (singular) linear line complex from one single regular (singular) linear line complex by a projective transformation. With respect to affine geometry one has to distinguish between two equivalence classes of regular linear line complexes, right winded and left winded complexes. In a projectively enclosed Euclidean space each linear complex is characterised by an Euclidean invariant, the so-called \emph{pitch} $p\in\mathds{R}$ (see \cite{potwal}, \cite{weiss}) and all proper lines of a regular linear line complex can be interpreted as the path normals of a helical motion.\\
If $(c,\bar{c})^\mathrm{T}\mathds{R}=((c_1,c_2,c_3),(c_4,c_5,c_6))^\mathrm{T}\mathds{R}$ is the Pl\"ucker coordinate vector of a linear line complex $C$, then we get its pitch p by
\begin{equation*}
p=\frac{c\cdot \bar{c}}{c^2}.
\end{equation*}
The axis of the helicoidal motion corresponding to the linear line complex $C$ is determined by
\[A=(a,\bar{a})^\mathrm{T}\mathds{R}=((a_1,a_2,a_3),(a_4,a_5,a_6))^\mathrm{T}\mathds{R}=(c,\bar{c}-pc)^\mathrm{T}\mathds{R}.\]
In case of a singular linear line complex the pitch equals $0$, $A=C$, and the complex contains all lines that are path normals of a one-parameter pure rotation or pure translation group of three-dimensional Euclidean space $\mathds{E}^3$.  
\end{rem}
\noindent The condition $\Omega(C,L)=0$ is exactly the condition for the intersection of lines in this case. With this knowledge we can identify the points of $\mathds{P}^5(\mathds{R})$ that are not contained by $M_2^4$ with regular linear line complexes. The points on $M_2^4$ correspond to singular linear line complexes. This is called the \emph{extended Klein mapping} $\hat{\mu}$, see \cite{weiss}.
\vspace{5pt}
\paragraph{Linear Line Complexes as five-blades}
A four-space in $\mathds{P}^5(\mathds{R})$ can be described as the outer product null space of a five-blade $\mathfrak{B}\in\bigwedge^5 V$. Using duality, the same four-space is defined by the inner product null space of the vector $\mathfrak{BJ}\in\bigwedge^1 V$. If the given vector is a null vector, the linear line complex is singular, else regular. The outer product null space of a vector and the inner product null space of its dual corresponds to the vector itself. Let $\mathfrak{v}=x_1e_1+x_2e_2+x_3e_3+x_4e_4+x_5e_5+x_6e_6$ be a general vector. Then its inner product null space is given by
\[\mathds{NI}(\mathfrak{v})=\left\lbrace \mathfrak{a}\in{\bigwedge}^1 V \mid x_1a_4+x_2a_5+x_3a_6+x_4a_1+x_5a_2+x_6a_3=0\right\rbrace,\]
with $\mathfrak{a}=a_1e_1+a_2e_2+a_3e_3+a_4e_4+a_5e_5+a_6e_6$. The same set of lines can be obtained as the outer product null space of the dual of $\mathfrak{v}$:
\[\mathds{NO}(\mathfrak{v}\mathfrak{J})=\left\lbrace \mathfrak{a}\in{\bigwedge}^1 V \mid x_1a_4+x_2a_5+x_3a_6+x_4a_1+x_5a_2+x_6a_3=0\right\rbrace,\]
with $\mathfrak{a}=a_1e_1+a_2e_2+a_3e_3+a_4e_4+a_5e_5+a_6e_6$.
Note, that it is not important from which side we multiply with the pseudo-scalar, since $\mathfrak{Jv}=-\mathfrak{vJ}$. We are working in a projective setting, and thus, the multiplication by $-1$ or by a real number does not change the geometric meaning of the object and its inner product or outer product null space.
\section{Transformations}\label{versorgroup}
To describe the action of the versor group we recall the definition of a versor, see \cite{perwass}.
\begin{defn}
A \emph{versor} $\mathfrak{a}$ is an algebra element that can be expressed as the $k$-fold geometric product of non-null vectors. Thus, $\mathfrak{a}=\mathfrak{v}_1\dots \mathfrak{v}_k$ with $\mathfrak{v}_i\in\bigwedge^1 V,\,i=1,\dots,k$.
\end{defn}
\noindent Transformations are applied with the so-called sandwich operator.
\begin{defn}
The product $\alpha(\mathfrak{a})\mathfrak{va}^{-1}$ is called the \emph{sandwich operator}, where $\alpha$ denotes the \emph{main involution} of the Clifford algebra. The main involution is defined by its action on generators
\[\alpha(e_{i_1\dots i_k})=(-1)^ke_{i_1\dots i_k},\]
with $0<i_1<\dots<i_k\leq n=\dim V$.
\end{defn}
\begin{rem}
For a vector $\mathfrak{v}\in\bigwedge^1 V$ the inverse element can be obtained by
\[\mathfrak{v}^{-1}=\frac{\mathfrak{v}^\ast}{\mathfrak{vv}^\ast}.\]
Thus, the inverse element differs from the conjugate element by a real factor. This is also true for a $k$-fold product of vectors.
\end{rem}
\noindent Since we are working in a homogeneous Clifford algebra model multiplication with a real factor does not change the geometric meaning of an algebra element. Therefore, we use the conjugate element instead of the inverse element. Hence, the sandwich operator that we use is given by:
\[\alpha(\mathfrak{a})\mathfrak{va}^\ast=\mathfrak{aa}^\ast(\alpha(\mathfrak{a})\mathfrak{va}^{-1}).\]
This operator does not involve an inverse, and therefore, it can also be applied if the element $\mathfrak{a}$ is not invertible.\\

\noindent If a vector corresponds to a line it is a null vector respectively a zero divisor, {\it i.e.}, it has no inverse. Applying the sandwich operator with an arbitrary element of the versor group to a zero divisor results again in a zero divisor since the zero divisors forms an ideal in the algebra. This means null vectors corresponding to lines in $\mathds{P}^3(\mathds{R})$ are mapped to null vectors corresponding to lines in $\mathds{P}^3(\mathds{R})$. The transformations induced by non-null vectors are reflections in $\mathds{R}^6$ as a model for $\mathds{P}^5(\mathds{R})$ that fix Klein's quadric.\\

\noindent The advantage of this  model is that transformations can be applied via the sandwich operator to every entity that can be represented as $\mathfrak{A}\in{\bigwedge}^k V$ in inner product or outer product null space representation. For example, we are able to apply the transformation to a four-blade that corresponds to a linear line congruence with one sandwich operator. The resulting element is again a four-blade corresponding to a linear line congruence that is of the same projective type (elliptic, parabolic, or hyperbolic) as the original one. At this point we recall a theorem from \cite[theorem 2.1.10]{potwal}.
\begin{thm}
Projective collineations and correlations of $\mathds{P}^3(\mathds{R})$ induce projective automorphisms of Klein's quadric, and Klein's quadric does not admit other projective automorphisms.
\end{thm}
\noindent The Clifford group of this Clifford algebra,{\it i.e.}, the projective automorphisms of Klein's quadric is generated by non-null vectors. Thus, we are interested in the action of non-null vectors on null vectors.
Let
\begin{align*}
\mathfrak{a}&=a_1 e_1 + a_2 e_2 + a_3 e_3 +a_4 e_4+a_5 e_5 +a_6 e_6,\\
\mathfrak{v}&=x_1e_1+ x_2 e_2 + x_3 e_3 +x_4 e_4+x_5 e_5 +x_6 e_6
\end{align*}
be two vectors with 
\[\mathfrak{aa}=a_1a_4+a_2a_5+a_3a_6\neq 0,\,\quad \mathfrak{vv}=v_1v_4+v_2v_5+v_3v_6=0.\]
The action of the sandwich operator $\alpha(\mathfrak{a})\mathfrak{v}\mathfrak{a}^\ast$ is linear on the vector space $\bigwedge^1 V$.
The matrix acting on $\mathds{P}^5(\mathds{R})$ can be represented by
\begin{equation}\label{nullpolaritycoll}
\mathrm{M}=\begin{pmatrix}
  k_1 & a_1a_5   & a_1a_6 & a_1a_1 & a_1a_2 & a_1a_3\\
 a_2a_4 & k_2    & a_2a_6 & a_2a_1 & a_2a_2 & a_2a_3\\
 a_3a_4 & a_3a_5 & k_3    & a_3a_1 & a_3a_2 & a_3a_3 \\
 a_4a_4 & a_4a_5 & a_4a_6 & k_4    & a_4a_2 & a_4a_3\\
 a_5a_4 & a_5a_5 & a_5a_6 & a_5a_1 & k_5    & a_5a_3\\
 a_6a_4 & a_6a_5 & a_6a_6 & a_6a_1 & a_6a_2 &  k_6 
\end{pmatrix},
\end{equation}
with
\begin{align*}
k_1&=-a_5a_2-a_6a_3,\quad &&k_2=-a_6a_3-a_4a_1,&&k_3=-a_4a_1-a_5a_2,\\
k_4&=-a_5a_2-a_6a_3,\quad &&k_5=-a_6a_3-a_4a_1,\quad &&k_6=-a_4a_1-a_5a_2.
\end{align*}
\noindent Naturally, we now ask for the corresponding projective mapping acting on $\mathds{P}^3(\mathds{R})$. Therefore, we examine the action of the collineation $\mathrm{M}$ on the space of lines. The action on the base lines
\begin{align*}
b_1&=(1:0:0:0:0:0),&&b_2=(0:1:0:0:0:0),&&b_3=(0:0:1:0:0:0),\\
b_4&=(0:0:0:1:0:0),&&b_5=(0:0:0:0:1:0),&&b_6=(0:0:0:0:0:1)
\end{align*}
is given by the columns of the matrix $\mathrm{M}$.
\begin{equation*}
h_1=\mathrm{M}b_1,\quad h_2=\mathrm{M}b_2,\quad h_3=\mathrm{M}b_3,\quad h_4=\mathrm{M}b_4,\quad h_5=\mathrm{M}b_5,\quad h_6=\mathrm{M}b_6.
\end{equation*}
If we look at the images of the three ideal lines $b_4,b_5,b_6$ we see, that the corresponding lines posses linear dependent direction vectors. This means that the three image lines $h_4,h_5,h_6$ intersect in the same ideal point, and therefore, they belong to a bundle of lines. Since the lines $b_4,b_5,b_6$ define a field of lines and the image is a bundle of lines, the mapping must be a correlation. With this knowledge we determine the action of $\mathrm{M}$ on the bundle of lines concurrent to the origin spanned by $b_1,b_2,b_3$. The image of this bundle of lines is a field of lines. To obtain the coordinates of the plane that carries the field of lines we intersect the image lines and get three points defining the plane in $\mathds{P}^3(\mathds{R})$
\begin{align*}
h_1\cap h_2&=s_1=(a_3,-a_5,a_4,0)^\mathrm{T},\\
h_1\cap h_3&=s_2=(a_2,a_6,0,-a_4)^\mathrm{T},\\
h_2\cap h_3&=s_3=(a_1,0,-a_6,a_5)^\mathrm{T}.
\end{align*}
The plane coordinates $p_1$ of the plane generated by these three points have to satisfy
\begin{equation}\label{planeeq}
\langle p_1,s_1\rangle=\langle p_1,s_2\rangle=\langle p_1,s_3\rangle=0.
\end{equation}
Solving Eq.\ \ref{planeeq} results in $p_1=\mathds{R}(0,a_4,a_5,a_6)^\mathrm{T}$ as image of the origin\linebreak $O=(1,0,0,0)^\mathrm{T}\mathds{R}$. We repeat this procedure for the bundle of lines spanned by $b_1,b_5,b_6$, {\it i.e.}, the bundle of lines concurrent to the ideal point\linebreak $X_u=(0,1,0,0)^\mathrm{T}\mathds{R}$.
\begin{align*}
h_1\cap h_5&=(a_2,a_6,0,-a_4)^\mathrm{T}, &&h_1\cap h_6=(a_3,-a_5,a_4,0)^\mathrm{T},\\
h_5\cap h_6&=(0,a_1,a_2,a_3)^\mathrm{T},  && p_2=\mathds{R}(-a_4,0,a_3,-a_2)^\mathrm{T}.
\end{align*}
Repeating the procedure for the bundle of lines concurrent to the ideal point $Y_u=(0,0,1,0)^\mathrm{T}\mathds{R}$ spanned by the lines $b_2,\,b_4,$ and $b_6$ yields
\begin{align*}
h_2\cap h_4&=(a_1,0,-a_6,a_5)^\mathrm{T}, &&h_2\cap h_6=(a_3,-a_5,a_4,0)^\mathrm{T},\\
h_4\cap h_6&=(0,a_1,a_2,a_3)^\mathrm{T}, &&p_3=\mathds{R}(-a_5,-a_3,0,a_1)^\mathrm{T}.
\end{align*}
Analogue computation for the bundle of lines spanned by $b_3,\,b_4,$ and $b_5$ concurrent to the ideal point $Z_u=(0,0,0,1)^\mathrm{T}\mathds{R}$ results in
\begin{align*}
h_3\cap h_4&=(a_1,0,-a_6,a_5)^\mathrm{T}, &&h_3\cap h_5=(a_2,a_6,0,-a_4)^\mathrm{T},\\
h_4\cap h_5&=(0,a_1,a_2,a_3)^\mathrm{T},  &&p_4=\mathds{R}(-a_6,a_2,-a_1,0)^\mathrm{T}.
\end{align*}
Note, that the correlation is not determined through $p_1,\dots,p_4$ since the scaling of each $p_i,\,i=1,\dots,4$ is not determined yet. Together with the matrix representation of a correlation see Eq. \eqref{collineationlinespace} the correlation in $\mathds{P}^3(\mathds{R})$ is fixed. It can be written in matrix form as
\begin{equation*}
m=\begin{pmatrix}
0 & -a_4 & -a_5 & -a_6\\ a_4 & 0 & -a_3 & a_2\\ a_5 & a_3 & 0 & -a_1\\ a_6 & -a_2 & a_1 & 0
\end{pmatrix}.
\end{equation*}
This correlation is a null polarity, see \cite{weiss:vor}. The determinant of this null polarity is calculated by
\begin{equation}\label{determinant}
\det m = (a_1a_4+a_2a_5+a_3a_6)^2={1\over 4}(\mathfrak{aa})^2.
\end{equation}
Thus, the square of a vector is related to the determinant of the corresponding null polarity.
\begin{rem}
Note, that this procedure works also for null vectors, because we do not use the inverse element to determine the sandwich product.
\end{rem}
\noindent With these preliminaries we are now able to prove the following theorem:
\begin{thm}\label{theoremKlein}
Each regular projective transformation, {\it i.e.}, regular correlation or regular collineation can be represented by the product of at most six null polarities.
\end{thm}
\begin{proof}
The group of automorphic collineations of $M_2^4$ and projective transformations of $\mathds{P}^3(\mathds{R})$ are isomorphic, see \cite{potwal}. Furthermore, the group of automorphisms of Klein's quadric can be described by the Pin group of $\clifford_{(3,3)}$, see \cite{gunn:kinematics} for the general case. The Clifford algebra model has at most grade six. Furthermore, the $k$-fold product of grade-1 elements, that correspond to null polarities generate the Pin or Spin group depending on whether $k$ is odd or even. To reach the maximum grade of six, we need a product of at least six grade-1 elements. That at most six elements are necessary follows from the Cartan-Dieudonn\'{e} theorem, cf. \cite{garling:cliffordalgebras}.
\end{proof}
\noindent All together we have that the versor group generated by reflections respectively non-null vectors corresponds to the group of regular projective transformations in $\mathds{P}^3(\mathds{R})$ that can be written as the product of null polarities. Furthermore, all transformations that are generated by the product of an even number of vectors are collineations,{\it i.e.}, elements of the Spin group when they are normalized. Transformations that are generated by an odd number of vectors correspond to correlations, and therefore, to the Pin group when they are normalized.
\begin{ex}\label{example1}
As example we present the action of a null polarity applied to a conic on Klein's quadric. Thus, we give three null vectors corresponding to Pl\"ucker coordinates $l_i,\,i=1,2,3$ of three lines
\[ \mathfrak{l}_1= e_1 + 2 e_6,\quad \mathfrak{l}_2= e_2 + 2 e_4,\quad \mathfrak{l}_3= e_3 + 2 e_5.\]
The corresponding two-space in $\mathds{P}^5(R)$ is represented as three-blade by
\[\mathfrak{P}=\mathfrak{l}_1\wedge\mathfrak{l}_2\wedge\mathfrak{l}_3=e_{123}\!+\!2e_{236}\!-\!2e_{134}\!-\!4e_{346}\!+\!2e_{125}\!
+\!4e_{256}\!+\!4e_{145}\!+\!8e_{456}.\]
The outer product null space of this entity can be calculated as
\[\mathds{NO}(\mathfrak{P})=\left\lbrace \alpha \left( \frac{1}{2}e_2+e_4\right) +\beta \left( \frac{1}{2}e_3+e_5\right)  +\gamma \left( \frac{1}{2}e_1+e_6\right)\mid \alpha,\beta,\gamma\in\mathds{R}   \right\rbrace. \]
Indeed, the outer product null space of $\mathfrak{P}$ spans the same two-space as $\left[ \mathfrak{l}_1,\mathfrak{l}_2,\mathfrak{l}_3\right]$.
To get the regulus defined by the three lines $L_i$ or their corresponding null vectors $\mathfrak{l}_i$ we have to find all null vectors in the two-space $\mathfrak{p}=\alpha\mathfrak{l}_1+\beta\mathfrak{l}_2+\gamma\mathfrak{l}_3$. Therefore, we compute
\begin{equation}\label{planecon}
\mathfrak{p}^2=4\beta\alpha+4\gamma\alpha+4\gamma\beta\stackrel{!}{=}0.
\end{equation}
The parameters $(\alpha:\beta:\gamma)$ can be interpreted as homogeneous coordinates in a projective plane $P_1^2$. Eq. \eqref{planecon} defines a conic contained in this plane. Now we apply an arbitrary null polarity given by \[\mathfrak{a}=a_1e_1+a_2e_2+a_3e_3+a_4e_4+a_5e_5+a_6e_6.\]
The outer product null space of the resulting three-blade $\mathfrak{P}'=\alpha(\mathfrak{a})\mathfrak{Pa}^\ast$ can be calculated as the set
\[\mathds{NO}(\mathfrak{P}')=\left\lbrace\alpha\mathfrak{v}_1+\beta\mathfrak{v}_2+\gamma\mathfrak{v}_3\mid \alpha,\beta,\gamma\in\mathds{R} \right\rbrace,\]
where the vectors $\mathfrak{v}_i$ are given by
\begin{align*}
\mathfrak{v}_1 &= \frac{2a_{5}a_{1}\!-\!a_{6}a_{5}\!-\!2a_{6}a_{1}\!+\!4a_{1}^2}{a_{5}a_{4}\!+\!a_{6}a_{4}\!+\!a_{6}a_{5}}e_1+\frac{a_{6}a_{5}\!+\!2a_{5}a_{2}\!+\! 4a_{1}a_{2}\!-\!2a_{4}a_{1}\!+\!a_{6}a_{4}}{a_{5}a_{4}\!+\!a_{6}a_{4}\!+\!a_{6}a_{5}}e_2\\
&+\frac{2a_{5}a_{3}\!-\!a_{5}^2\!+\!4a_{1}a_{3}\!-\!2a_{5}a_{1}}{a_{5}a_{4}\!+\!a_{6}a_{4}\!+\!a_{6}a_{5}}+2e_4,\\
\mathfrak{v}_2 &= \frac{4a_{1}a_{2}\!+\!2a_{6}a_{1}\!-\!a_{6}^2\!-\!2a_{6}a_{2}}{a_{5}a_{4}\!+\!a_{6}a_{4}\!+\!a_{6}a_{5}}e_1 +\frac{2a_{6}a_{2}\!-\!2a_{4}a_{2}\!-\!a_{6}a_{4}\!+\!4a_{2}^2}{a_{5}a_{4}\!+\!a_{6}a_{4}\!+\!a_{6}a_{5}}e_2\\
&+\frac{a_{5}a_{4}\!+\!2a_{6}a_{3}\!+\!a_{6}a_{4}\!+\!4a_{3}a_{2}\!-\!2a_{5}a_{2}}{a_{5}a_{4}\!+\!a_{6}a_{4}\!+\!a_{6}a_{5}}e_3+2e_5,\\
\mathfrak{v}_3 &=\frac{2a_{4}a_{1}\!+\!4a_{1}a_{3}\!+\!a_{6}a_{5}\!-\!2a_{6}a_{3}\!+\!a_{5}a_{4}}{a_{5}a_{4}\!+\!a_{6}a_{4}\!+\!a_{6}a_{5}}e_1+\frac{2a_{4}a_{2}\!-\!2a_{4}a_{3}\!+\!4a_{3}a_{2}\!-\!a_{4}^2}{a_{5}a_{4}\!+\!a_{6}a_{4}\!+\!a_{6}a_{5}}e_2\\
&+ \frac{2a_{4}a_{3}\!-\!a_{5}a_{4}\!+\!4a_{3}^2\!-\!2a_{5}a_{3}}{a_{5}a_{4}\!+\!a_{6}a_{4}\!+\!a_{6}a_{5}}e_3+2e_6.
\end{align*}
We search for null vectors contained in the plane $\mathfrak{p}'(\alpha,\beta,\gamma)=\alpha\mathfrak{v}_1+\beta\mathfrak{v}_2+\gamma\mathfrak{v}_3$. Therefore, we calculate $\mathfrak{p}'\mathfrak{p}'=0$. This leads to the following quadratic equation
\begin{align*}
&((-2a_{4}a_{3}+a_{5}a_{4}+2a_{6}a_{3}+a_{6}a_{4}+2a_{4}a_{2}+8a_{3}a_{2}-2a_{5}a_{2}-a_{4}^2)\beta\\
&+(8a_{1}a_{3}-2a_{6}a_{3}-2a_{5}a_{1}+2a_{4}a_{1}+2a_{5}a_{3}+a_{6}a_{5}-a_{5}^2+a_{5}a_{4})\alpha)
\gamma\\
&+(-2a_{4}a_{1}+2a_{5}a_{2}+2a_{6}a_{1}+a_{6}a_{4}+a_{6}a_{5}+8a_{1}a_{2}-2a_{6}a_{2}-a_{6}^2)\alpha\beta\\
&+(2a_{5}a_{1}-a_{6}a_{5}-2a_{6}a_{1}+4a_{1}^2)+(2a_{4}a_{3}-a_{5}a_{4}+4a_{3}^2-2a_{5}a_{3})
\gamma^2\alpha^2\\
&+(2a_{6}a_{2}-2a_{4}a_{2}-a_{6}a_{4}+4a_{2}^2)\beta^2=0.
\end{align*}
This equation defines a conic in the projective plane $P_1^2$. Every point on this conic delivers the parameters for a null vector that corresponds to a line in $\mathds{P}^3(\mathds{R})$.
\end{ex}
\subsection{Collineations as Spin Group}
We are interested in the relationship between regular projective mappings and elements of the versor group of $\clifford_{(3,3)}$. The general approach (cf.\cite{perwass}) does not work for this model. Therefore, we develop a more line geometric approach by the examination of the action of versors on null three-blades corresponding to two-spaces that are entirely contained in $M_2^4$.
A general element $\mathfrak{g}\in\clifford_{(3,3)}^+$ corresponding to a collineation is given by
\begin{align*}
\mathfrak{g} =&g_{1}e_0\!+\!g_{2}e_{12}\!+\!g_{3}e_{13} \!+\!g_{4}e_{14}\!+\!g_{5}e_{15}\!+\!g_{6}e_{16}\!+\!g_{7}e_{23}
\!+\!g_{8}e_{24}\!+\!g_{9}e_{25}\!+\!g_{10}e_{26}\\
+&g_{11}e_{34}\!+\!g_{12}e_{35}\!+\!g_{13}e_{36}\!+\!g_{14}e_{45}\!+\!g_{15}e_{46}\!+\!g_{16}e_{56}\!+\!g_{17}e_{1234}\!+\!g_{18}e_{1235}\\
+&g_{19}e_{1236}\!+\! g_{20}e_{1245}\!+\!g_{21}e_{1246}\!+\!g_{22}e_{1256} \!+\!g_{23}e_{1345}\!+\!g_{24}e_{1346}\!+\!g_{25}e_{1356}\\
+&g_{26}e_{1456}\!+\!g_{27}e_{2345}\!+\!g_{28}e_{2346}\!+\!g_{29}e_{2356}\!+\!g_{30}e_{2456}\!+\!g_{31}e_{3456}\!+\!g_{32}e_{123456}.
\end{align*}
The conditions that this element is the product of invertible vectors is obtained by
\begin{equation}\label{constraint1}
\alpha(\mathfrak{g})\mathfrak{vg}^\ast\in{\bigwedge}^1 V \mbox{ for all } \mathfrak{v}\in{\bigwedge}^1 V.
\end{equation}
This results in 36 quadratic equations that occur as coefficients of the grade-5 element $e_J$, where $J\subset\lbrace1,2,3,4,5,6\rbrace$ with $|J|=5$.
The 36 quadratic equations defined above define a pseudo algebraic variety in $\mathds{P}^{31}(\mathds{R})$ that can be interpreted as image space of $\mathrm{Spin}_{(3,3)}$, see \cite{klaw}.
Points of $\mathds{P}^3(\mathds{R})$ can be embedded in the homogeneous Clifford algebra as null three-blades that correspond to two-spaces on Klein's quadric. These two-spaces correspond to bundles of lines, {\it i.e.}, all lines concurrent to a point. The point contained by all lines of a bundle can be determined and is described uniquely by the null three-blade. Therefore, we examine the action of an arbitrary versor corresponding to a collineation on a null three-blade corresponding to a bundle of lines. This results in a null three-blade that corresponds to a bundle of lines again. Afterwards we compute the point concurrent to all lines contained by the bundle of lines to get the image of the start point under the transformation.\\

\noindent We start with the bundle of lines concurrent to the a general point\linebreak $P=(y_0,y_1,y_2,y_3)^\mathrm{T}\mathds{R}$. The Pl\"ucker coordinates of three lines containing $P$ are computed as
\[l_1\!=\!(y_0\!:\!0\!:\!0\!:\!0\!:\!y_3\!:\!-y_2),\,\,l_2\!=\!(0\!:\!y_0\!:\!0\!:\!-y_3\!:\!0\!:\!y_1),\,\,l_3\!=\!(0\!:\!0\!:\!y_0\!:\!y_2\!:\!-y_1\!:\!0).\]
The null three-blade corresponding to the two-space spanned by $l_1,l_2$, and $l_3$ is obtained by \begin{align*}
\mathfrak{b}&=(y_0e_1+y_3e_5-y_2e_6) \wedge (y_0e_2-y_3e_4+y_1e_6) \wedge (y_0e_3+y_2e_4-y_1e_5)\\
&=y_0^3e_{123}+y_3y_0^2e_{235}-y_2y_0^2e_{236}+y_3y_0^2e_{134}+y_3^2y_0e_{345}-y_2y_3y_0e_{346}\\
&-y_0^2y_1e_{136}+y_3y_1y_0e_{356}+y_2y_0^2e_{124}+y_2y_3y_0e_{245}-y_2^2y_0e_{246}-y_0y_1y_2e_{146}\\
&-y_0^2y_1e_{125}+y_0y_1y_2e_{256}+y_3y_1y_0e_{145}+y_0y_1^2e_{156}.
\end{align*}
We apply the sandwich operator with the arbitrary element $\mathfrak{g}\in\clifford_{(3,3)}^+$.
The outer product null space of the null three-blade $\mathfrak{b}'=\alpha(\mathfrak{g})\mathfrak{bg}^\ast$ is obtained by
\[\mathds{NO}(\mathfrak{b}')=\left\lbrace \alpha \mathfrak{v}_1+\beta \mathfrak{v}_2+\gamma \mathfrak{v}_3| \alpha,\beta,\gamma\in\mathds{R}\right\rbrace,\]
where
\begin{align*}
\mathfrak{v}_1 &= \big(y_0(g_{1}-g_{32}-g_{20}+g_{4}-g_{29}+g_{13}+g_{9}-g_{24})\\
&+2y_1(g_{17}+g_{7})+2y_2(g_{18}-2g_{3})+2y_3(g_{2}+2g_{19})\big)e_1\\
&-\big(2y_0(g_{31}+g_{14})+2y_1(g_{27}-g_{11})-2y_2(g_{12}y_2+g_{23})\\
&+y_3(g_{9}-g_{1}+g_{32}-g_{29}+g_{4}-g_{13}+g_{20}-g_{24})\big)e_5\\
&-\big(+2y_0(g_{15}-g_{30})+2y_1(g_{8}+g_{28})+2y_3(g_{10}+g_{21})\\
&+y_2(g_{9}-g_{24}-g_{4}+g_{1}-g_{32}+g_{20}+g_{29}-g_{13})\big)e_6,\\
\mathfrak{v}_2 &= \big(y_0(g_{1}-g_{29}+g_{13}+g_{9}-g_{24}-g_{32}-g_{20}+g_{4})\\
&+2y_1(g_{17}+g_{7})+2y_2(g_{18}-g_{3})+2y_3(g_{19}+g_{2})\big)e_2\\
&+\big(2y_0(g_{14}+g_{31})+2y_1(g_{27}-g_{11})-2y_2(g_{12}y_2+g_{23})\\
&+y_3(g_{9}-g_{1}+g_{32}-g_{29}+g_{4}-g_{13}+g_{20}-g_{24})\big)e_4\\
&-\big(2y_0(g_{16}+g_{26})-2y_2(g_{5}+g_{25})+2y_3(g_{22}-g_{6})\\
&+y_1(g_{13}+g_{32}-g_{4}+g_{29}+g_{9}-g_{24}-g_{20}-g_{1})\big)e_6,\\
\mathfrak{v}_3 &= \big(y_0(g_{1}-g_{29}+g_{13}+g_{9}-g_{24}-g_{32}-g_{20}+g_{4})\\
&+2y_1(g_{17}+g_{7})+2y_2(g_{18}-g_{3})+2y_3(g_{2}+g_{19})\big)e_3\\
&+\big(2y_0(g_{15}-g_{30})+2y_1(g_{8}+g_{28})+2y_3(g_{10}+g_{21})\\
&+y_2(g_{1}-g_{24}-g_{4}-g_{32}+g_{9}+g_{20}+g_{29}-g_{13})\big)e_4\\
&+\big(2y_0(g_{16}+g_{26})+2y_3(g_{22}-g_{6})-2y_2(g_{5}+g_{25})\\
&+y_1(g_{13}+g_{29}+g_{32}-g_{4}+g_{9}-g_{24}-g_{20}-g_{1})\big)e_5.
\end{align*}
The null vectors $\mathfrak{v}_1,\mathfrak{v}_2$, and $\mathfrak{v}_3$ span a bundle of lines concurrent to the image of the origin under the collineation corresponding to $\mathfrak{g}$. To get the point we have to intersect two lines of the bundle of lines. Therefore, we change the model and transfer the null vectors corresponding to $\mathfrak{v}_1,\,\mathfrak{v}_2$ to two-blades $\mathfrak{L}_1,\,\mathfrak{L}_2$ of the Grassmann algebra $\mathcal{G}(\mathds{P}^3)$, cf. \cite{gunn:onthehomogeneousmodel}. The representation of a line $L=(p_{01}\!:\!p_{02}\!:\!p_{03}\!:\!p_{23}\!:\!p_{31}\!:\!p_{12})$ in $\mathcal{G}(\mathds{P}^3)$ is 
\begin{equation}\label{transfer}
L=p_{01}e_{12}+p_{02}e_{13}+p_{03}e_{14}+p_{23}e_{34}-p_{31}e_{24}+p_{12}e_{23}.
\end{equation}
Note that due to historic conversion the coefficient of $e_{24}$ is multiplied by $-1$.
To get the intersection point of these two lines we take an arbitrary point $X=x\mathds{R}\in\mathds{P}^3(\mathds{R})$ with representation $\mathfrak{X}=x_0e_1+x_1e_2+x_2e_3+x_3e_4\in{\bigwedge}^1\mathcal{G}(\mathds{P}^3)$ and compute the incidence condition with both lines $\mathfrak{L}_1$ and $\mathfrak{L}_2$:
\[\mathfrak{L}_1 \wedge \mathfrak{X}=0,\quad \mathfrak{L}_2 \wedge \mathfrak{X}=0.\]
This results in a system of linear equations for $x_0,x_1,x_2,x_3$ with solution
\begin{align*}
x_0&=y_0(g_{4}\!+\!g_{9}\!-\!g_{29}\!+\!g_{13}\!-\!g_{20}\!-\!g_{32}\!-\!g_{24}\!+\!g_{1})\!+\!2y_1(g_{17}\!+\!g_{7})\!+\!2y_2(g_{18}\!-\!g_{3})\\
&+2y_3(g_{2}\!+\!g_{19}),\\
x_1&=-\!2y_0(g_{16}\!+\!g_{26})\!+\!y_1(g_{1}\!-\!g_{9}\!-\!g_{13}\!-\!g_{29}\!+\!g_{20}\!+\!g_{4}\!-\!g_{32}\!+\!g_{24})\\
&+2y_2(g_{5}\!+\!g_{25})\!+\!2y_3(g_{6}\!-\!g_{22}),\\
x_2&=2y_0(g_{15}\!-\!g_{30}) \!+\! 2y_1(g_{8}\!+\!g_{28})\\
&+y_2(g_{29}\!-\!g_{24}\!-\!g_{4}\!+\!g_{1}\!+\!g_{9}\!+\!g_{20}\!-\!g_{32}\!-\!g_{13})\!+\!2y_3(g_{10}\!+\!g_{21}),\\
x_3&=-\!2y_0(g_{31}\!+\!g_{14})\!+\! 2y_1(g_{11}\!-\!g_{27})\\
&+2y_2(g_{12}\!+\!g_{23})\!+\!y_3(g_{1}\!-\!g_{32}\!-\!g_{20}\!-\!g_{9}\!+\!g_{13}\!+\!g_{29}\!-\!g_{4}\!+\!g_{24}).
\end{align*}
If we rewrite $X$ as product of a matrix $\mathrm{M}$ with the vector $(y_0,y_1,y_2,y_3)^\mathrm{T}$ we get
\begin{equation}\label{collrep}
(x_0,x_1,x_2,x_3)^\mathrm{T}=\mathrm{M}\cdot(y_0,y_1,y_2,y_3)^\mathrm{T},
\end{equation}
with
\begin{equation*}
\mathrm{M}=\begin{pmatrix}
 k_1 & 2(g_{7}+g_{17}) &  2(g_{18}-g_{3}) & 2(g_{19}+g_{2})\\
-2(g_{26}+g_{16}) & k_2 & 2(g_{5}+g_{25}) & 2(g_{6}-g_{22})\\ 
 2(g_{15}-g_{30}) & 2(g_{8}+g_{28}) &  k_3 &2(g_{21}+2g_{10})\\
-2(g_{31}+g_{14}) & 2(g_{11}-g_{27}) & 2(g_{23}+g_{12}) & k_4
\end{pmatrix},
\end{equation*}
where
\begin{align*}
k_1&=g_{1}-g_{20}-g_{24}-g_{32}-g_{29}+g_{9}+g_{4}+g_{13},\\
k_2&=g_{24}-g_{9}+g_{20}-g_{13}-g_{32}+g_{1}+g_{4}-g_{29},\\
k_3&=g_{1}-g_{13}-g_{32}-g_{4}+g_{29}+g_{9}-g_{24}+g_{20},\\
k_4&=g_{24}+g_{13}+g_{29}+g_{1}-g_{4}-g_{9}-g_{20}-g_{32}.
\end{align*}
Therefore, we have found the correspondence between regular collineations and elements of $\clifford_{(3,3)}^+$. For a given versor $\mathfrak{g}\in\clifford_{(3,3)}^+$  corresponding to a collineation we can determine the representation as $4\times 4$ matrix. If we start with a $4\times 4$ matrix $\mathrm{A}$ with coefficients $(a_{ij}),\,i,j=0,\dots,3$ representing a collineation we have to solve a system of 16 linear and 36 quadratic equations to get the corresponding algebra representation. The linear equations are derived with the help of the matrix representation determined above and result in
\begin{align}\label{systemcollineation}
a_{01}&=2(g_{7}+g_{17}), &&a_{02}=2(g_{18}-g_{3}), &&a_{03}=2(g_{19}+g_{2}),\\
a_{10}&=-2(g_{26}+g_{16}), &&a_{12}=2(g_{5}+g_{25}), &&a_{13}= 2(g_{6}-g_{22}),\nonumber\\
a_{20}&=2(g_{15}-g_{30}), &&a_{21}=2(g_{8}+g_{28}), &&a_{23}=2(g_{21}+2g_{10}),\nonumber\\
a_{30}&=-2(g_{31}+g_{14}), &&a_{31}=2(g_{11}-g_{27}), &&a_{32}=2(g_{23}+g_{12}),\nonumber
\end{align} 
\begin{align*}
a_{00}&=g_{1}-g_{20}-g_{24}-g_{32}-g_{29}+g_{9}+g_{4}+g_{13},\\
a_{11}&=g_{24}-g_{9}+g_{20}-g_{13}-g_{32}+g_{1}+g_{4}-g_{29},\\
a_{22}&=g_{1}-g_{13}-g_{32}-g_{4}+g_{29}+g_{9}-g_{24}+g_{20},\\
a_{33}&=g_{24}+g_{13}+g_{29}+g_{1}-g_{4}-g_{9}-g_{20}-g_{32}.\\
\end{align*}
These 16 equations have to be solved with the constraint equations derived Eq. \eqref{constraint1}. For both cases $\mathfrak{gg}^\ast=1$ and $\mathfrak{gg}^\ast=-1$ the resulting system of 16 linear and 36 quadratic equations possesses an unique solution. It can be solved analytically with the help of a computer algebra system. We give an example:
\begin{ex}
Let $\mathrm{K}\in PGL(\mathds{P}^3(\mathds{R}))$ be given by
\[K=\begin{pmatrix}
1 & 0 & 3 & 0\\1 & 1 & 0 & 1\\1 & 2 & 1 & 0\\1 & 1 & 2 & 1
\end{pmatrix}.\]
To get a versor $\mathfrak{g}\in\clifford_{(3,3)}^+$ corresponding to this collineation we have to solve the system \eqref{systemcollineation}.
\begin{align}\label{system}
&2(g_{7}+g_{17})=0, && 2(g_{18}-g_{3})=3, && 2(g_{19}+g_{2})=0,\\
-&2(g_{26}+g_{16})=1, &&2(g_{5}+g_{25})=0, && 2(g_{6}-g_{22})=1,\nonumber\\
&2(g_{15}-g_{30})=1, &&2(g_{8}+g_{28})=2, &&2(g_{21}+2g_{10})=0,\nonumber\\
-&2(g_{31}+g_{14})=1, && 2(g_{11}-g_{27})=1, &&2(g_{23}+g_{12})=2,\nonumber
\end{align} 
\begin{align*}
&g_{1}-g_{20}-g_{24}-g_{32}-g_{29}+g_{9}+g_{4}+g_{13}=1,\\
&g_{24}-g_{9}+g_{20}-g_{13}-g_{32}+g_{1}+g_{4}-g_{29}=1,\\
&g_{1}-g_{13}-g_{32}-g_{4}+g_{29}+g_{9}-g_{24}+g_{20}=1,\\
&g_{24}+g_{13}+g_{29}+g_{1}-g_{4}-g_{9}-g_{20}-g_{32}=1.\\
\end{align*}
We have two possibilities to guarantee that the resulting versor is in the Spin group, {\it i.e.}, $\mathfrak{gg}^\ast=1$ or $\mathfrak{gg}^\ast=-1$. We compute both solutions and start with the constraint equations implied by Eq. \eqref{constraint1} and $\mathfrak{gg}^\ast=1$.
The corresponding Spin group element has the form:
\begin{align*}
\mathfrak{g_+}={1\over 8\sqrt{2}} \big(&7 e_{0}\! +\!6 e_{12}\!-\!6 e_{13}\! +\! e_{14}\!- \!2 e_{15}\!- \!6 e_{23} \!+\!6 e_{24}\! - \!e_{25}\! -\!2 e_{26}\! +\!2 e_{34}\\
+&6 e_{35}\!-\!5 e_{36}\! -\!4 e_{45}\! +\!2 e_{46}\! +\!6 e_{1234}\! -\! 4 e_{56}\! +\!6 e_{1235}\! -\!6 e_{1236}\\
-&5 e_{1245}\!+\!2 e_{1246}\!-\!4 e_{1256}\! +\!2 e_{1345}\!-\!e_{1346}\! +\!2 e_{1356}\! -\!2 e_{2345}\\
+&2 e_{2346} \!+\! e_{2356}\!-\!2 e_{2456} \!-\! e_{123456}\big).
\end{align*}
If we demand that $\mathfrak{gg}^\ast=-1$ the resulting Spin group element is computed as
\begin{align*}
\mathfrak{g}_-={1\over 8\sqrt{2}}\big(&e_{0}\!-\!6 e_{12}\!-\!6 e_{13}\!-\!e_{14}\!+\!2 e_{15}\!+\!4 e_{16}\!+\!6 e_{23}\!+\!2 e_{24}\!+\!e_{25}\!+\!2 e_{26}\!+\!2 e_{34}\\
+&2 e_{35}\!+\!5 e_{36}\!+\!2 e_{46}\!-\!6 e_{1234}\!+\!6 e_{1235}\!+\!6 e_{1236}\!+\!5 e_{1245}\!-\!2 e_{1246}\\
+&6 e_{1345}\!+\!e_{1346}\!-\!2 e_{1356}\!-\!4 e_{1456}\!-\!2 e_{2345}\!+\!6 e_{2346}\!-\!e_{2356}\\
-&2 e_{2456}\!-\!4 e_{3456}\!-\!7 e_{123456}
\big).
\end{align*}
Both elements $\mathfrak{g}_+$ and $\mathfrak{g}_-$ correspond to the same collineation. With Eq. \eqref{collrep} we get the matrix $\mathrm{K}$ back with the coefficients of $\mathfrak{g}_+$ and $\mathfrak{g}_-$.
\end{ex}
\subsection{Correlations as Pin Group}
We know that grade-$1$ elements respectively their action on null vectors correspond to null polarities. Thus, correlations are elements of $\clifford_{(3,3)}^-$. To transfer a $4\times 4$ matrix representing a correlation into the Clifford algebra we use a similar approach as we used for collineations. An arbitrary field of lines contained in the plane determined by its plane coordinates $X=\mathds{R}(x_0,x_1,x_2,x_3)^\mathrm{T}$ corresponds to the outer product null space of a null three-blade. We apply an arbitrary versor $\mathfrak{h}$ contained in $\clifford_{(3,3)}^-$ to this null three-blade. The image is again a null three-blade whose outer product null space is a bundle of lines. Thus, we compute the point concurrent to all lines of this bundle of lines and describe the action of $\mathfrak{h}$ as product of a matrix with a vector. A general element $\mathfrak{h}\in\clifford_{(3,3)}^-$ is given by
\begin{align*}
\mathfrak{h} &= h_{1}e_{1}\!+\!h_{2}e_{2}\!+\!h_{3}e_{3}\!+\!h_{4}e_{4}\!+\!h_{5}e_{5}\!+\!h_{6}e_{6} \!+\!h_{7}e_{123}\!+\!h_{8}e_{124}\!+\!h_{9}e_{125}\!+\!h_{10}e_{126}\\
&+h_{11}e_{134}\!+\!h_{12}e_{135}\!+\!h_{13}e_{136}\!+\!h_{14}e_{145}\!+\!h_{15}e_{146}\!+\!h_{16}e_{156}\!+\!h_{17}e_{234}\!+\!h_{18}e_{235}\\
&+h_{19}e_{236}\!+\!h_{20}e_{245}\!+\!h_{21}e_{246}\!+\!h_{22}e_{256}\!+\!h_{23}e_{345}\!+\!h_{24}e_{346}\!+\!h_{25}e_{356}\!+\!h_{26}e_{456}\\
&+h_{27}e_{12345}\!+\!h_{28}e_{12346}\!+\!h_{29}e_{12356}\!+\!h_{30}e_{12456}\!+\!h_{31}e_{13456}\!+\!h_{32}e_{23456}.
\end{align*}
If this element shall be a versor we have constraint equations that can be derived from \begin{equation}\label{constraintcorrelation}
\alpha(\mathfrak{h})\mathfrak{vh}^\ast\in{\bigwedge}^1V \mbox{ for all } \mathfrak{v}\in{\bigwedge}^1 V.
\end{equation}
We start with an arbitrary plane given by its plane coordinates\linebreak $Y=\mathds{R}(y_0,y_1,y_2,y_3)^\mathrm{T}$. Pl\"ucker coordinates of three lines contained in this plane can be determined by
\[l_1=(0\!:\!y_3\!:\!-y_2\!:\!y_0\!:\!0\!:\!0),\,\,l_2=(-y_3\!:\!0,y_1\!:\!0\!:\!y_0\!:\!0),\,\,l_3=(y_2\!:\!-y_1\!:\!0\!:\!0\!:\!0\!:\!y_0).\]
The null three-blade corresponding to the plane $\left[ l_1,l_2,l_3\right]$ whose outer product null space is the field of lines defined by $l_1,l_2$ and $l_3$ is computed as
\begin{align*}
\mathfrak{b}&=(y_3 e_2 -y_2 e_3 +y_0 e_4)\wedge(-y_3 e_1 + y_1 e_3 +y_0 e_5)\wedge(y_2 e_1 -y_1 e_2 + y_0 e_6)\\
&= -y_0y_1y_2e_{134}+y_0y_3y_2e_{125}-y_2^2y_0e_{135}+y_0^2y_2e_{145}+y_0y_3y_1e_{124}+y_0y_1^2e_{234}\\
&+y_0y_1y_2e_{235}-y_0^2y_1e_{245}+y_3^2y_0e_{126}-y_0y_3y_2e_{136}+y_0^2y_3e_{146}+y_0y_3y_1e_{236}\\
&-y_0^2y_1e_{346}+y_0^2y_3e_{256}-y_0^2y_2e_{356}+y_0^3e_{456}.
\end{align*}
We apply the sandwich operator with the general element $\mathfrak{h}\in\clifford_{(3,3)}^-$.
Since a versor from the odd part of the algebra corresponds to a correlation the null three-blade $\mathfrak{b}'=\alpha(\mathfrak{h})\mathfrak{bh}^\ast$ corresponds to a bundle of lines. Thus, we can apply the same procedure that we used for collineations to determine the image point, {\it i.e.}, the point concurrent to all lines of the bundle of lines. Therefore, we compute the outer product null space of $\mathfrak{b}'$
\[\mathds{NO}(\mathfrak{b}')=\left\lbrace \alpha \mathfrak{v}_1+\beta \mathfrak{v}_2+\gamma \mathfrak{v}_3| \alpha,\beta,\gamma\in\mathds{R}\right\rbrace,\]
where
\begin{align*}
\mathfrak{v}_1 = &-\big(2y_0h_{7}+y_1(h_{29}-h_{1}-h_{9}-h_{13})+y_2(h_{8}-h_{2}-h_{19}-h_{28})\\
&+y_3(h_{11}+h_{27}+h_{18}-h_{3})\big)e_1\\
&-\big(y_0(h_{3}+h_{11}-h_{27}+h_{18})+y_1(h_{25}+h_{14}+h_{31}-h_{5})\\
&+y_2(h_{4}+h_{32}-h_{24}+h_{20})+2y_3h_{23}\big)e_5\\
&+\big(y_0(h_{2}+h_{8}-h_{19}+h_{28})+y_1(h_{6}+h_{22}-h_{15}+h_{30})-2y_2h_{21}\\
&+y_3(h_{20}-h_{4}-h_{32}-h_{24})\big)e_6,\\
\mathfrak{v}_2 = &-\big(2y_0h_{7}+y_1(h_{29}-h_{1}-h_{9}-h_{13})+y_2(h_{8}-h_{2}-h_{28}-h_{19})\\
&+y_3(h_{27}+h_{18}-h_{3}+h_{11})\big)e_2\\
&+\big(y_0(h_{3}+h_{11}+h_{18}-h_{27})+y_1(h_{25}+h_{14}+h_{31}-h_{5})\\
&+y_2(h_{4}+h_{32}-h_{24}+h_{20})+2y_3h_{23}\big)e_4\\
&-\big(y_0(h_{1}-h_{29}-h_{13}-h_{9})+2h_{16}y_1+y_2(h_{22}-h_{6}-h_{15}-h_{30})\\
&+y_3(h_{25}+h_{5}-h_{31}+h_{14})\big)e_6,\\
\mathfrak{v}_3 = &-\big(2y_0h_{7}+y_1(h_{29}-h_{13}-h_{1}-h_{9})+y_2(h_{8}-h_{2}-h_{28}-h_{19})\\
&+y_3(h_{11}-h_{3}+h_{27}+h_{18})\big)e_3\\
&-\big(y_0(h_{2}+h_{8}-h_{19}+h_{28})+y_1(h_{6}+h_{22}-h_{15}+h_{30})-2y_2h_{21}\\
&+y_3(h_{20}-h_{4}-h_{32}-h_{24})\big)e_4\\
&+\big(y_0(h_{1}-h_{29}-h_{13}-h_{9})+2h_{16}y_1+y_2(h_{22}-h_{6}-h_{15}-h_{30})\\
&+y_3(h_{5}+h_{14}+h_{25}-h_{31})\big)e_5.
\end{align*}
These three null vectors correspond to three points on Klein's quadric that span the bundle of lines defined by the outer product null space of $\mathfrak{b}'$. The point concurrent to these three lines is computed with the help of the exterior algebra $\mathcal{G}(\mathds{P}^3)$. Hence, we transfer the null vectors $\mathfrak{v}_1,\,\mathfrak{v}_2$ to two-blades $\mathfrak{L}_1,\,\mathfrak{L}_2$ of $\mathcal{G}(\mathds{P}^3)$ with Eq. \eqref{transfer}.
The intersection point of these two lines is the image of the plane $Y=\mathds{R}(y_0,y_1,y_2,y_3)^\mathrm{T}$. A general point $X=(x_0,x_1,x_2,x_3)^\mathrm{T}\mathds{R}$ is written as element of the exterior algebra $\mathfrak{X}=x_0e_1+x_1e_2+x_2e_3+x_3e_4\in \mathcal{G}(\mathds{P}^3)$. The point $X$ is incident with the lines $L_1$ and $L_2$ if $\mathfrak{L}_1\wedge \mathfrak{X}=0$ and $\mathfrak{L}_2\wedge \mathfrak{X}=0$.
This results in a system of linear equations for $x_0,x_1,x_2,x_3$ with solution
\begin{align*}
x_0&=-\!2y_0h_{7}\!+\!y_1(h_{1}\!+\!h_{9}\!+\!h_{13}\!-\!h_{29})\!+\!y_2(h_{2}\!+\!h_{19}\!+\!h_{28}\!-\!h_{8})\\
&+\!y_3(h_{3}\!-\!h_{18}\!-\!h_{27}\!-\!h_{11}),\\
x_1&= y_0(h_{13}\!-\!h_{1}\!+\!h_{29}\!+\!h_{9})\!-\!2y_1h_{16}\!+\!y_2(h_{6}\!+\!h_{30}\!+\!h_{15}\!-\!h_{22})\!\\
&+y_3(h_{31}\!-\!h_{25}\!-\!h_{14}\!-\!h_{5}),\\
x_2&= y_0(h_{19}\!-\!h_{2}\!-\!h_{8}\!-\!h_{28})\!+\!y_1(h_{15}\!-\!h_{30}\!-\!h_{6}\!-\!h_{22})\!+\!2h_{21}y_2\!\\
&+y_3(h_{4}\!-\!h_{20}\!+\!h_{32}\!+\!h_{24}),\\
x_3&= y_0(h_{27}\!-\!h_{3}\!-\!h_{11}\!-\!h_{18})\!+\!y_1(h_{5}\!-\!h_{31}\!-\!h_{25}\!-\!h_{14})\\
&+\!y_2(h_{24}\!-\!h_{4}\!-\!h_{20}\!-\!h_{32})\!-\!2h_{23}y_3.
\end{align*}
If we rewrite $X$ as a product of a matrix $\mathrm{C}$ with $Y=(y_0,y_1,y_2,y_3)^\mathrm{T}$ we get:
\begin{equation}\label{correlation}
\begin{pmatrix}
x_0\\x_1\\x_2\\x_3
\end{pmatrix}=\begin{pmatrix}
c_{00} & c_{01} & c_{02} & c_{03}\\
c_{10} & c_{11} & c_{12} & c_{13}\\
c_{20} & c_{21} & c_{22} & c_{23}\\
c_{30} & c_{31} & c_{32} & c_{33}
\end{pmatrix}\begin{pmatrix}
y_0\\y_1\\y_2\\y_3
\end{pmatrix}\mbox{, with}
\end{equation}
\begin{align}
c_{01}&\!=\!h_{1}\!+\!h_{9}\!+\!h_{13}\!-\!h_{29},\!\! && c_{02}\!=\!h_{2}\!+\!h_{19}\!+\!h_{28}\!-\!h_{8},\!\! && c_{03}\!=\!h_{3}\!-\!h_{18}\!-\!h_{27}\!-\!h_{11},\nonumber\\
c_{10}&\!=\!h_{13}\!-\!h_{1}\!+\!h_{29}\!+\!h_{9},\!\! && c_{12}\!=\!h_{6}\!+\!h_{30}\!+\!h_{15}\!-\!h_{22},\!\! && c_{13}\!=\!h_{31}\!-\!h_{25}\!-\!h_{14}\!-\!h_{5},\nonumber\\
c_{20}&\!=\!h_{19}\!-\!h_{2}\!-\!h_{8}\!-\!h_{28},\!\! && c_{21}\!=\!h_{15}\!-\!h_{30}\!-\!h_{6}\!-\!h_{22},\!\! &&  c_{23}\!=\!h_{4}\!-\!h_{20}\!+\!h_{32}\!+\!h_{24},\nonumber\\
c_{30}&\!=\!h_{27}\!-\!h_{3}\!-\!h_{11}\!-\!h_{18},\!\! && c_{31}\!=\!h_{5}\!-\!h_{31}\!-\!h_{25}\!-\!h_{14},\!\! && c_{32}\!=\!h_{24}\!-\!h_{4}\!-\!h_{20}\!-\!h_{32}, \nonumber
\end{align}
and
\[c_{00}=-2h_{7},\quad c_{11}=-2h_{16},\quad  c_{22}=2h_{21},\quad  c_{33}=-2h_{23}.\]
The matrix $\mathrm{C}$ describes the correspondence between regular correlations and elements of $\clifford_{(3,3)}^-$. For a given versor that corresponds to a correlation, we can compute the matrix representation by Eq. \eqref{correlation}. Furthermore, we can determine the versor corresponding to a regular correlation represented by a $4\times 4$ matrix $\mathrm{A}$ if we solve the system of 16 linear equation given by
\begin{align*}
a_{01}&\!=\!h_{1}\!+\!h_{9}\!+\!h_{13}\!-\!h_{29},\!\! &&a_{02}\!=\!h_{2}\!+\!h_{19}\!+\!h_{28}\!-\!h_{8},\!\! &&a_{03}\!=\!h_{3}\!-\!h_{18}\!-\!h_{27}\!-\!h_{11},\\
a_{10}&\!=\!h_{13}\!-\!h_{1}\!+\!h_{29}\!+\!h_{9},\!\!  &&a_{12}\!=\!h_{6}\!+\!h_{30}\!+\!h_{15}\!-\!h_{22},\!\! && a_{13}\!=\!h_{31}\!-\!h_{25}\!-\!h_{14}\!-\!h_{5},\\
a_{20}&\!=\!h_{19}\!-\!h_{2}\!-\!h_{8}\!-\!h_{28},\!\! &&a_{21}\!=\!h_{15}\!-\!h_{30}\!-\!h_{6}\!-\!h_{22},\!\!  && a_{23}\!=\!h_{4}\!-\!h_{20}\!+\!h_{32}\!+\!h_{24},\\
a_{30}&\!=\!h_{27}\!-\!h_{3}\!-\!h_{11}\!-\!h_{18},\!\! &&a_{31}\!=\!h_{5}\!-\!h_{31}\!-\!h_{25}\!-\!h_{14},\!\! && a_{32}\!=\!h_{24}\!-\!h_{4}\!-\!h_{20}\!-\!h_{32},
\end{align*}
\begin{equation}\label{systemcorrelation}
a_{00}=-2h_{7},\quad a_{11}=-2h_{16},\quad a_{22}=2h_{21},\quad a_{33}=-2h_{23}
\end{equation}
with respect to the constraint equations, see Eq. \eqref{constraintcorrelation}. With the help of a computer algebra system it can be verified that this system possesses an unique solution for each of the cases $\mathfrak{hh}^\ast=\pm1$.
\begin{ex}
As example we take again the matrix
\[K=\begin{pmatrix}
1 & 0 & 3 & 0 \\1 & 1 & 0 & 1\\1 & 2 & 1 & 0\\1 & 1 & 2 & 1
\end{pmatrix},\]
but now the matrix describes a correlation. To get the corresponding versor we have to solve the system of linear equations
\begin{align*}
&h_{1}\!+\!h_{9}\!+\!h_{13}\!-\!h_{29}\!=\!0,\!\! &&h_{2}\!+\!h_{19}\!+\!h_{28}\!-\!h_{8}\!=\!3,\!\! &&h_{3}\!-\!h_{18}\!-\!h_{27}\!-\!h_{11}\!=0,\\
&h_{13}\!-\!h_{1}\!+\!h_{29}\!+\!h_{9}\!=\!1,\!\!  &&h_{6}\!+\!h_{30}\!+\!h_{15}\!-\!h_{22}\!=\!0,\!\! && h_{31}\!-\!h_{25}\!-\!h_{14}\!-\!h_{5}\!=\!1,\\
&h_{19}\!-\!h_{2}\!-\!h_{8}\!-\!h_{28}\!=1,  &&h_{15}\!-\!h_{30}\!-\!h_{6}\!-\!h_{22}\!=\!2,  && h_{4}\!-\!h_{20}\!+\!h_{32}\!+\!h_{24}\!=\!0,\\
&h_{27}\!-\!h_{3}\!-\!h_{11}\!-\!h_{18}\!=\!1,\!\! &&h_{5}\!-\!h_{31}\!-\!h_{25}\!-\!h_{14}\!=\!1,\!\! && h_{24}\!-\!h_{4}\!-\!h_{20}\!-\!h_{32}=2,
\end{align*}
\[-2h_{7}=1,\quad -2h_{16}=1,\quad 2h_{21}=1,\quad -2h_{23}=1.\]
If we solve this system of linear equations with respect to the constraints given by Eq. \eqref{constraintcorrelation} and $\mathfrak{hh}^\ast=1$ the resulting versor is given by
\begin{align*}
\mathfrak{h}_+={1\over 8}\big(&e_{1}\!+\!e_{2}\!+\!e_{3}\!+\!2 e_{4}\!-\!4 e_{6}\!-\!2 e_{123}\!-\!3 e_{124}\!+\!e_{125}\!+\!4 e_{126}\!-\!e_{134}\!+\!e_{136}\\
-&2 e_{145}\!-\!2 e_{146}\!-\!2 e_{156}\!-\!2 e_{234}\!-\!e_{235}\!+\!5 e_{236}\!+\!2 e_{246}\!-\!4 e_{245}\!-\!2 e_{345}\\
-&\!6 e_{256}\!+\!4 e_{456}\!-\!2 e_{356}\!+\!3 e_{12345}\!+\!3 e_{12346}\!+\!3 e_{12356}\!-\!6 e_{23456}
\big).
\end{align*}
For $\mathfrak{hh}^\ast=-1$ we get
\begin{align*}
\mathfrak{h}_-={1 \over 8} \big( -&3 e_{1}\!+\!3 e_{2}\!-\!3 e_{3}\!-\!6 e_{4}\!-\!2 e_{123}\!-\!5 e_{124}\!+\!e_{125}\!-\!4 e_{126}\!-\!e_{134}\!+\!e_{136}\\
-&2 e_{145}\!+\!6 e_{146}\!-\!2 e_{156}\!+\!2 e_{234}\!-\!e_{235}\!+\!3 e_{236}\!+\!2 e_{246}\!-\!2 e_{345}\!+\!2 e_{256}\\
+&4 e_{346}\!-\!4 e_{456}\!-\!2 e_{356}\!-\!e_{12345}\!+\!e_{12346}\!-\!e_{12356}\!-\!4 e_{12456}\!+\!2 e_{23456}
\big).
\end{align*}
Note that both versors $\mathfrak{h}_+$ and $\mathfrak{h}_-$ correspond to the same correlation. The correlation can be computed with Eq. \eqref{correlation}.
\end{ex}
\begin{rem}
With the knowledge which elements correspond to collineations and which correspond to correlations we can apply a transformation to a pencil of lines to study the action of a mapping on a point. Since a general projective transformation maps pencils of lines to pencils of lines, we could not distinguish between collineations and correlations. Therefore, we used bundles of lines and fields of lines to study the action of a general projective transformation represented as element of $\clifford_{(3,3)}$.
\end{rem}
\subsection{Singular projective Transformations}
Now we examine singular projective transformations. Every versor can be written as the geometric product of grade-$1$ elements corresponding to null polarities. The square of a grade-1 element (see Eq. \eqref{square}) is related to the determinant of the corresponding null polarity, see Eq. \eqref{determinant}. Thus, we assume that one of these null polarities is a singular one, and therefore, its determinant is equal to zero. We define:
\begin{defn}
An element $\mathfrak{g}\in\clifford_{(3,3)}$ with
$\mathfrak{g}=\mathfrak{v}_1\dots\mathfrak{v}_k$
with $k\leq 6$ and $\mathfrak{v}_i\in\bigwedge^1 V$ is called a \emph{null versor} if at least one $\mathfrak{v}_i$ squares to zero.
\end{defn}
\noindent Constraint equations can be derived from
\[\alpha(\mathfrak{g})\mathfrak{vg}^\ast\in{\bigwedge}^1 V \mbox{ for all } \mathfrak{v}\in{\bigwedge}^1 V,\quad  \mathfrak{g}\in\clifford_{(3,3)}^+\]
for a singular collineation and
\[\alpha(\mathfrak{h})\mathfrak{vh}^\ast\in{\bigwedge}^1 V \mbox{ for all } \mathfrak{v}\in{\bigwedge}^1 V, \quad  \mathfrak{h}\in\clifford_{(3,3)}^-\]
for a singular correlation. The matrix representations for general collineations \eqref{collrep} and correlations \eqref{correlation} are also valid for the singular case. The system of linear equations \eqref{systemcollineation} with the constraint equations implied by $\alpha(\mathfrak{g})\mathfrak{vg}^\ast\in{\bigwedge}^1 V \mbox{ for all } \mathfrak{v}\in{\bigwedge}^1 V$ with $\mathfrak{g}\in\clifford_{(3,3)}^+$ can not be solved for a singular collineation. For a singular correlation the system of linear equations \eqref{systemcorrelation} with the constraint equations implied by $\alpha(\mathfrak{h})\mathfrak{vh}^\ast\in{\bigwedge}^1 V \mbox{ for all } \mathfrak{v}\in{\bigwedge}^1 V$ with $\mathfrak{h}\in\clifford_{(3,3)}^-$ has also no solution. Nevertheless, we are able to write singular correlations or collineations as null versors if we know the null polarities that generate them. These null polarities can be transferred to vectors in the homogeneous Clifford algebra model, and therefore, the whole correlation or collineation can be transferred to the algebra model.

\section{Conclusion}
We presented a method to transfer general projective transformations acting in three-dimensional projective space to elements of the homogeneous Clifford algebra model $\clifford_{(3,3)}$. All entities known from line geometry occur naturally in this model and can be transformed projectively by the application of the sandwich operator. Furthermore, we showed that the sandwich action of non-null vectors corresponds to regular null polarities, {\it i.e.}, correlations that are involutions as the basic elements building up the group of regular projective transformations. Moreover, the sandwich action of null vectors corresponds to singular null polarities.


\subsection*{Acknowledgment}
This work was supported by the research project "Line Geometry for Lightweight Structures", funded by the DFG (German Research Foundation) as part of the SPP 1542.


\begin{thebibliography}{10}
\bibitem{blaschke:kinematikundquaternionen}
   Blaschke, W.: \textit{Kinematik und Quaternion}.
   VEB Deutscher Verlag der Wissenschaften, Berlin,\ 1960.
\bibitem{gallier:cliffordalgebrascliffordgroups}
   Gallier, J.: \textit{Clifford Algebras, Clifford Groups, and a Generalization of the Quaternions: The Pin and Spin Groups}.
    2008, arxiv.org, available at \url{http://arxiv.org/abs/0805.0311}.
\bibitem{garling:cliffordalgebras}
   Garling, D.J.H.: \textit{Clifford Algebras: An Introduction}.
   Camb. Uni. Press,\ 2011.
\bibitem{giering:vorlesungenueberhoeheregeometrie}
   Giering, O.: \textit{Vorlesungen \"uber h\"ohere Geometrie}.
   Vieweg,\ 1982.
\bibitem{gunn:kinematics}
   Gunn, C.: \textit{Geometry, Kinematics, and Rigid Body Mechanics in Cayley-Klein Geometries}.
    PhD Thesis, TU Berlin,\ 2011.   
\bibitem{gunn:onthehomogeneousmodel}
   Gunn, C.: \textit{On the Homogeneous Model of Euclidean Geometry}.
   In: Guide to Geometric Algebra in Practice, ch. 15, 297--327, Springer, 2011.
\bibitem{klaw}{Klawitter, D. \and Hagemann, M.:} \textit{Kinematic mappings for Cayley-Klein geometries via Clifford algebras.} Beitr\"age zur Algebra und Geometrie / Contributions to Algebra and Geometry \textbf{54} (2013), 737--761.
\bibitem{hongbo}
   Li, H. and Zhang, L.: \textit{Line Geometry in Terms of the Null Geometric Algebra over {$R^{3,3}$}, and Application to the Inverse Singularity Analysis of Generalized {Stewart} Platforms}. Guide to Geometric Algebra in Practice, 253--272, Springer, 2011.
\bibitem{onishchik}
  Onishchik, A.L. \and Sulanke, R.: \textit{Projective and Cayley-Klein geometries}.
   Springer, Berlin,\ 2006.
\bibitem{perwass}
   Perwass, C.B.U. \textit{Geometric Algebra with Applications in Engineering} Springer, Berlin and Heidelber, 2009.
\bibitem{porteous:cliffordalgebrasandtheclassicalgroups}
  Porteous, Ian R.: \textit{Clifford Algebras and the Classical Groups}.
   Cambridge University Press,\ 1995.
\bibitem{potwal}
    Pottmann, H \and Wallner, J. \textit{Computational Line Geometry} Springer, Berlin and New York, 2001
\bibitem{weiss:vor}
      Weiss, G.: \textit{Vorlesungen aus Liniengeometrie} Unpublished manuscript, Institute of Geometry, Vienna University of Technology.
\bibitem{weiss}
   Weiss, G.: \textit{Zur euklidischen Liniengeometrie I-IV} Sitzungsberichte der \"Osterreichischen Akademie der Wissenschaften, 1978-1982.
\end{thebibliography}
\end{document}